\documentclass{birkjour}
 \newtheorem{theorem}{Theorem}[section]
 \newtheorem{corollary}[theorem]{Corollary}
 \newtheorem{lemma}[theorem]{Lemma}
 
 \theoremstyle{definition}
 \newtheorem{definition}[theorem]{Definition}
 \theoremstyle{remark}
 \newtheorem{rem}[theorem]{Remark}
 \newtheorem*{example}{Example}
 \numberwithin{equation}{section}

\begin{document}

%
%
%
%
%
%
%
%
%

\title[Bi-slant submersions]
 {On bi-slant submersions in complex geometry}

\author[C. Sayar]{Cem Sayar$^1$}
\address{$^1$Istanbul Technical University\\ Faculty of Science and Letters,\\ Department of Mathematics\\ 34469, Maslak /\.{I}stanbul Turkey}
\email{sayarce@itu.edu.tr}
\author[M. A. Akyol]{Mehmet Akif Akyol$^2$}
\address{$^2$Bingol University\\ Faculty of Arts and Sciences,\\ Department of Mathematics\\ 12000, Bing\"{o}l, Turkey}
\email{mehmetakifakyol@bingol.edu.tr}
\author[R. Prasad]{Rajendra Prasad$^3$}
\address{$^3$Lucknow University\\ Department of Mathematics and Astronomy\\ 226007, Uttar Pradesh, Lucknow, India}
\email{rp.manpur@rediffmail.com}






\date{January 1, 2004}

\begin{abstract}
In the present paper, we introduce bi-slant submersions from almost Hermitian manifolds onto Riemannian manifolds
as a generalization of invariant, anti-invariant, semi-invariant, slant, semi-slant and hemi-slant Riemannian submersions.
We mainly focus on bi-slant submersions from Kaehler manifolds. We provide a proper example of bi-slant submersion, investigate the geometry of foliations
determined by vertical and horizontal distributions, and obtain the geometry of leaves of these distributions. Moreover, we obtain curvature relations between the base space, the total space and the fibres, and find geometric implications of
these relations.
\end{abstract}

\maketitle
\section{Introduction}

The notion of a slant submanifold was introduced by B.-Y. Chen in \cite{Chen0} and first results on slant submanifolds
were collected in his book \cite{Chen2}. After he defined that notion, many geometers were inspired by that fact and have obtained many results on the notion in the different total space.
As a generalization of the notion, J. L. Cabrerizo et. al. defined the notion of bi-slant submanifold in \cite{Cabre} and see also \cite{Carri}.

On the other hand, as an analogue of isometric immersion (Riemannian submanifold), the notion of Riemannian submersion was first introduced by B. O'Neill \cite{O} and A. Gray \cite{Gra} between two Riemannian manifolds. This notion has some aplications in physics and in mathematics. More precisely, Riemannian submersions have  applications in supergravity and superstring theories \cite{IV1,M}, Kaluza-Klein theory \cite{BL,IV} and the Yang-Mills theory \cite{BL1,W1}. B. Watson \cite{Wat} considered submersions between almost Hermitian manifolds by taking account of almost complex structure of total manifold.
In this case,  the vertical and horizontal distributions are invariant.  Afterwards,  almost Hermitian submersions have been extensively studied different subclasses of almost Hermitian manifolds, for example; see \cite{Fa}.

Inspried by B. Watson's article, B. \c{S}ahin introduced anti invariant submersions from
almost Hermitian manifolds onto Riemannian manifolds \cite{Sah}. This notion has opened a new original and effective area in the theory of Riemannian submersions. 
That paper has been a source of inspiration to so many geometers. For example, as a special case of anti-invariant submersion, Lagrangian submersion was studied by
H. M. Tastan \cite{Ta}. Later, several new types of Riemannian submersions were defined and studied such as semi-invariant submersion \cite{Akyol4,cem,Sa}, slant submersion \cite{Er, Gun, Gun1, Sa1}, hemi-slant submersion \cite{Ta3}, semi-slant submersion \cite{Akyol3, Park}, pointwise slant submersion \cite{Lee, Se}, quasi bi-slant submersion \cite{Prasad}, conformal slant submersion \cite{Akyol0,Akyol1} and
conformal semi-slant submersion \cite{Akyol2}. Also, these kinds of submersions were
considered in different kinds of structures such as cosymplectic, Sasakian, Kenmotsu, nearly Kaehler, almost product, para-contact, and et al. Recent developments in the
theory of submersion can be found in the book \cite{baykit}. 

Recently, the first author of the paper and et.al. define Generic submersion in the sense of G. B. Ronsse (see: \cite{Rons}) for the complex context in \cite{Cemp}. We are motivated to fill a gap in the literature by giving the notion of bi-slant submersions in which the fibres consist of two slant distributions. In the present paper, as a special case of the above notion and generalization of invariant, anti-invariant, semi-invariant, slant, semi-slant and hemi-slant Riemannian submersions we introduce bi-slant submersion and investigate the geometry of base space, the total space and the fibres.

The paper is organized as follows. Section 2 includes preliminaries. In section 3 contains the definition of bi-slant submersions, a proper example, the geometry of foliations
determined by vertical and horizontal distributions and the geometry of leaves of these distributions. The last section of this paper includes curvature relations between the base space, the total space and the fibres, and find geometric implications of these relations.

\section{Riemannian submersions}
In this section,  we give necessary background for
Riemannian submersions.\\

Let $(M,g)$ and $(N,g_{\text{\tiny$N$}})$ be
Riemannian manifolds, where $\dim(M)$ is greater than $\dim(N)$. A surjective mapping
$\pi:(M,g)\rightarrow(N,g_{N})$ is called a
\emph{Riemannian submersion}
\cite{O} if\\

\textbf{(S1)} $\pi$ has maximal rank, and \\

\textbf{(S2)} $\pi_{*}$, restricted to $\ker\pi_{*}^{\bot},$ is a linear
isometry.\\

In this case, for each $q\in N$, $\pi^{-1}(q)$ is a $k$-dimensional
submanifold of $M$ and called a \emph{fiber}, where $k=\dim(M)-\dim(N).$
A vector field on $M$ is called \emph{vertical} (resp.
\emph{horizontal}) if it is always tangent (resp. orthogonal) to
fibers. A vector field $X$ on $M$ is called \emph{basic} if $X$ is
horizontal and $\pi$-related to a vector field $X_{*}$ on $N,$ i.e.,
$\pi_{*}X_{p}=X_{*\pi(p)}$ for all $p\in M.$ We will denote by
$\mathcal{V}$ and $\mathcal{H}$ the projections on the vertical
distribution $\ker\pi_{*}$, and the horizontal distribution
$\ker\pi_{*}^{\bot},$ respectively. As usual, the manifold $(M,g)$ is called \emph{total manifold} and
the manifold $(N,g_{N})$ is called \emph{base manifold} of the submersion $\pi:(M,g)\rightarrow(N,g_{N})$.
The geometry of Riemannian
submersions is characterized by O'Neill's tensors $\mathcal{T}$ and
$\mathcal{A}$, defined as follows:
\begin{equation}\label{testequationn}
\mathcal{T}_{U}{V}=\mathcal{V}\nabla_{\mathcal{V}{U}}\mathcal{H}{V}+\mathcal{H}\nabla_{\mathcal{V}{U}}\mathcal{V}{V},
\end{equation}
\begin{equation}\label{testequationnn}
\mathcal{A}_{U}{V}=\mathcal{V}\nabla_{\mathcal{H}{U}}\mathcal{H}{ V}+\mathcal{H}\nabla_{\mathcal{H}{U}}\mathcal{V}{V}
\end{equation}
for any vector fields ${U}$ and ${V}$ on $M,$ where $\nabla$ is the
Levi-Civita connection of $g$. It is easy to see
that $\mathcal{T}_{{U}}$ and $\mathcal{A}_{{U}}$ are skew-symmetric
operators on the tangent bundle of $M$ reversing the vertical and
the horizontal distributions. We now summarize the properties of the
tensor fields $\mathcal{T}$ and $\mathcal{A}$. Let $V,W$ be  vertical
and $X,Y$ be horizontal vector fields on $M$, then we have
\begin{equation}\label{testequation111}
\mathcal{T}_{V}W=\mathcal{T}_{W}V,
\end{equation}
\begin{equation}\label{testequation00}
\mathcal{A}_{X}Y=-\mathcal{A}_{Y}X=\frac{1}{2}\mathcal{V}[X,Y].
\end{equation}
On the other hand, from (\ref{testequationn}) and (\ref{testequationnn}), we obtain
\begin{equation}\label{testequation09}
\nabla_{V}W=\mathcal{T}_{V}W+\hat{\nabla}_{V}W,
\end{equation}
\begin{equation}\label{testequation11}
\nabla_{V}X=\mathrm{T}_{V}X+\mathcal{H}\nabla_{V}X,
\end{equation}
\begin{equation}\label{testequation}
\nabla_{X}V=\mathcal{A}_{X}V+\mathcal{V}\nabla_{X}V,
\end{equation}
\begin{equation}\label{testequation123}
\nabla_{X}Y=\mathcal{H}\nabla_{X}Y+\mathcal{A}_{X}Y,
\end{equation}
where $\hat{\nabla}_{V}W=\mathcal{V}\nabla_{V}W$. If $X$ is basic
\[\mathcal{H}\nabla_{V}X=\mathcal{A}_{X}V.\]
\begin{rem}\label{remark1}
In this paper, we will assume all horizontal vector fields as basic vector fields.
\end{rem}
It is not difficult to observe that $\mathcal{T}$ acts on the fibers
as the second fundamental form while $\mathcal{A}$  acts on the
horizontal distribution and measures of the obstruction to the
integrability of this distribution. For  details on Riemannian
submersions, we refer to O'Neill's paper \cite{O} and to
the book \cite{Fa}.

\section{Bi-slant Submersions}

A manifold $M$ is called an \textit{almost Hermitian manifold} \cite{Yan} if it admits a tensor field $J$ of type (1,1) on itself such that, for any $X,Y \in TM$
\begin{equation} \label{e9}
J^{2}=-I,\quad g(X,Y)=g(JX,JY).
\end{equation}
An almost Hermitian manifold $M$ is called \textit{Kaehler manifold} \cite{Yan} \\if $\forall X,Y \in TM$,
\begin{equation} \label{e10}
(\nabla_{X}J)Y=0,
\end{equation}
where $\nabla$ is the Levi-Civita connection with respect to the Riemannian metric $g$
and $I$ is the identity operator on the tangent bundle $TM$.\\
\begin{definition}\label{dfnbislant}
Let $(M,g,J)$ be a Kaehler manifold and $(N,g_{\text{\tiny$N$}})$ be a Riemannian manifold. A Riemannian submersion $\pi : (M,g,J)\rightarrow (N,g_{N})$ is called a \textit{bi-slant submersion}, if there are two slant distributions $\mathcal{D}^{\theta_{1}}\subset ker\pi_{*}$ and $\mathcal{D}^{\theta_{2}}\subset ker\pi_{*}$ such that
\begin{equation}\label{eqnbislant1}
 ker\pi_{*}=\mathcal{D}^{\theta_{1}}\oplus \mathcal{D}^{\theta_{2}},
\end{equation}
where, $\mathcal{D}^{\theta_{1}}$ and $\mathcal{D}^{\theta_{2}}$ has slant angles $\theta_{1}$ and $\theta_{2}$, respectively. 
\end{definition}

Suppose the dimension of distribution of $\mathcal{D}^{\theta_{1}}$ (resp. $\mathcal{D}^{\theta_{2}}$) is $m_1$ (resp. $m_2$). Then we easily see the following particular cases.
\begin{enumerate}
\item[(a)] If $m_1=0$ and $\theta_2=0$, then $\pi$ is an invariant submersion.

\item[(b)]If $m_1=0$ and $\theta=\frac{\pi}{2},$ then $\pi$ is an anti-invariant submersion.

\item[(c)] If $m_1\neq m_2\neq0,$ $\theta_1=0$ and $\theta_2=\frac{\pi}{2},$ then $\pi$ is a semi-invariant submersion.

\item[(d)] If $m_1=0$ and $0<\theta_2<\frac{\pi}{2},$ then $\pi$ is a proper slant submersion.

\item[(e)] If $m_1\neq m_2\neq0,$ $\theta_1=0$ and $0<\theta_2<\frac{\pi}{2},$ then $\pi$ is a semi-slant submersion.

\item[(e)] If $m_1\neq m_2\neq0,$ $\theta_1=\frac{\pi}{2}$ and $0<\theta_2<\frac{\pi}{2},$ then $\pi$ is a hemi-slant submersion.
\end{enumerate}

If each slant angles are different from either zero or $\frac{\pi}{2}$, then the bi-slant submersion is called a \textit{proper bi-slant submersion}. Now, we present a non-trivial example of bi-slant submersions and demonstrate that the method presented in this paper is effective.

\begin{rem}
In present paper, we assume bi-slant submersion as proper bi-slant submersion i.e. slant angles are from either zero or $\frac{\pi}{2}$.
\end{rem}

\begin{example}
Let $\mathbb{R}^{8}$ be $8-dimensional$ Euclidean space. $\mathbb{R}^{8},J,g$ is a Kaehler manifold with Euclidean metric $g$ on $\mathbb{R}^{8}$ and canonical complex structure $J$.
Consider the map $\pi : \mathbb{R}^{8} \rightarrow \mathbb{R}^{4}$ with
\begin{equation*}
  \pi(x_{1},x_{2},...x_{8})\mapsto (\frac{-x_{1}+x_{4}}{\sqrt{2}},-x_{2},\frac{-\sqrt{3}x_{5}+x_{8}}{2},-x_{6}).
\end{equation*}
Then, we have the Jacobian matrix of $\pi$ has rank $4$. That means $\pi$ is a submersion. So, with some calculations we observe that
\begin{equation*}
  ker\pi_{*}=\mathcal{D}^{\theta_{1}}\oplus \mathcal{D}^{\theta_{2}},
\end{equation*}
where
\begin{equation*}
  \mathcal{D}^{\theta_{1}}=span \{V_{1}=\frac{1}{\sqrt{2}}(\partial x_{1}+\partial x_{4}), V_{2}=\partial x_{3}\}
\end{equation*}
and
\begin{equation*}
  \mathcal{D}^{\theta_{2}}=span \{V_{3}=\frac{1}{2}\partial x_{5}+\frac{\sqrt{3}}{2}\partial x_{8}, V_{4}=\partial x_{7}\}.
\end{equation*}
Moreover, the slant angle of $\mathcal{D}^{\theta_{1}}$ is $\theta_{1}=\frac{\pi}{4}$ and the slant angle of $\mathcal{D}^{\theta_{2}}$ is $\theta_{2}=\frac{\pi}{3}$.
\end{example}
Let $\pi : (M,g,J)\rightarrow (N,g_{N})$ be a bi-slant submersion from a Kaehlerian manifold $M$ onto a Riemannian manifold $N$. Then, for any $V \in ker\pi_{*}$, we put
\begin{equation}\label{decompvervec}
 JV=PV+FV,
\end{equation}
where $PV \in ker\pi_{*}$ and $FV \in ker\pi_{*}^{\perp}$. Also, for any $\xi \in ker\pi_{*}^{\perp}$, we put
\begin{equation}\label{decomphorvec}
 J\xi=\phi \xi +\omega \xi,
\end{equation}
where $\phi \xi \in ker\pi_{*}$ and $\omega \xi \in ker\pi_{*}^{\perp}$. In this case, the horizontal distribution $ker\pi_{*}^{\perp}$ can be decomposed as follows
\begin{equation}\label{eqnbislant2}
 ker\pi_{*}^{\perp}=F\mathcal{D}^{\theta_{1}}\oplus F\mathcal{D}^{\theta_{2}} \oplus \mu,
\end{equation}
where $\mu$ is the orthogonal complementary of $F\mathcal{D}^{\theta_{1}}\oplus F\mathcal{D}^{\theta_{2}}$ in $ ker\pi_{*}^{\perp}$, and it is invariant with respect to the complex structure $J$. \\
By using \eqref{decompvervec} and \eqref{decomphorvec}, we obtain the followings.
\begin{lemma}
Let $\pi$ be a bi-slant submersion from a Kaehlerian manifold $(M,g,J)$ onto a Riemannian manifold $(N,g_{N})$. Then, we have
\begin{equation*}
   \textbf{(a)}P\mathcal{D}^{\theta_{1}}\subset \mathcal{D}^{\theta_{1}},\quad
 \textbf{ (b)}P\mathcal{D}^{\theta_{2}}\subset \mathcal{D}^{\theta_{2}},\quad
  \textbf{(c)}\phi \mu=\{0\},\quad
  \textbf{(d)}\omega \mu = \mu.
\end{equation*}
\end{lemma}
With the help of \eqref{e9}, \eqref{decompvervec} and \eqref{decomphorvec} we obtain the following Lemma.
\begin{lemma}\label{general}
Let $\pi$ be a bi-slant submersion from a Kaehlerian manifold $(M,g,J)$ onto a Riemannian manifold $(N,g_{N})$. Then, we have
\begin{equation*}
  \textbf{(a)}\, P^{2}X=-\cos^{2}\theta_{1}X, \quad \textbf{(b)}\, P^{2}U=-\cos^{2}\theta_{2}U,
\end{equation*}
\begin{equation*}
\textbf{(c)}\, \phi FX=-\sin ^{2}\theta_{1}X, \quad \textbf{(d)}\, \phi FU=-\sin ^{2}\theta_{2}U,
\end{equation*}
\begin{equation*}
  \textbf{(e)}\, P^{2}X+\phi FX=-X, \quad \textbf{(f)}\, P^{2}U+\phi FU=-U,
\end{equation*}
\begin{equation*}
  \textbf{(g)}\, FPX+\omega FX=0, \quad \textbf{(h)}\, FPU+\omega FU=0,
\end{equation*}
for any vector field $X \in \mathcal{D}^{\theta_{1}}$ and $U \in \mathcal{D}^{\theta_{2}}$.
\end{lemma}
We investigate the relation between complex structure $J$ and O'Neill tensors $\mathcal{T}$ and $\mathcal{A}$.
\begin{lemma}\label{lemmagenel}
Let $\pi$ be a bi-slant submersion from a Kaehlerian manifold $(M,g,J)$ onto a Riemannian manifold $(N,g_{N})$. Then, we have
\begin{equation}\label{eq1}
  \phi \mathcal{T}_{X}Y+P\hat{\nabla}_{X}Y=\hat{\nabla}_{X}PY+\mathcal{T}_{X}FY,
\end{equation}
\begin{equation}\label{eq2}
  \omega \mathcal{T}_{X}Y+F\hat{\nabla}_{X}Y=\mathcal{T}_{X}PY+\mathcal{A}_{FY}X,
\end{equation}
\begin{equation}\label{eq3}
  P\mathcal{T}_{X}\xi+\phi \mathcal{A}_{\xi}X=\hat{\nabla}_{X}\phi \xi+\mathcal{T}_{X}\omega \xi,
\end{equation}
\begin{equation}\label{eq4}
  F\mathcal{T}_{X}\xi+\omega \mathcal{A}_{\xi}X=\mathcal{T}_{X}\phi \xi+\mathcal{A}_{\omega \xi}X,
\end{equation}
\begin{equation}\label{eq5}
\phi \mathcal{H}\nabla_{\xi}\eta+P\mathcal{A}_{\xi}\eta=\mathcal{V}\nabla_{\xi}\phi \eta+\mathcal{A}_{\xi}\eta,
\end{equation}
\begin{equation}\label{eq6}
\omega \mathcal{H}\nabla_{\xi}\eta+F\mathcal{A}_{\xi}\eta=\mathcal{A}_{\xi}\phi \eta+\mathcal{H}\nabla_{\xi}\omega \eta,
\end{equation}
for any $U,V \in ker\pi_{*}$ and $\xi, \eta \in ker\pi_{*}^{\perp}$.
\end{lemma}
\begin{proof}
  Let $U$ and $V$ be in $ker\pi_{*}$. Since $M$ is Kaehlerian manifold, we have $J\nabla_{U}V=\nabla_{U}JV$. From \eqref{testequation09}, \eqref{testequation11}, \eqref{decompvervec} and \eqref{decomphorvec}, we obtain
  \begin{eqnarray*}
  J\nabla_{U}V&=&\nabla_{U}PV+\nabla_{U}FV\\
  \Rightarrow J(\mathcal{T}_{U}V+\hat{\nabla}_{U}V)&=&\mathcal{T}_{U}PV+\hat{\nabla}_{U}PV\\
  &+&\mathcal{T}_{U}FV+\mathcal{H}\mathcal{\nabla}_{U}FV.
  \end{eqnarray*}
  \begin{eqnarray*}
  \Rightarrow \phi \mathcal{T}_{U}V+\omega \mathcal{T}_{U}V+P\hat{\nabla}_{U}V+F\hat{\nabla}_{U}V&=& \mathcal{T}_{U}PV+\hat{\nabla}_{U}PV\\
  &+&\mathcal{T}_{U}FV+\mathcal{H}\mathcal{\nabla}_{U}FV.
  \end{eqnarray*}
Then, in the view of Remark \ref{remark1}, considering the vertical and horizontal parts of the last equation gives us \eqref{eq1} and \eqref{eq2}. For the rest of the equations, the same way could be applied.
\end{proof}
Now, we obtain equations which mean Gauss and Weingarten equations for bi-slant submersions.
\begin{lemma}\label{GauWei}
Let $\pi$ be a bi-slant submersion from a Kaehlerian manifold $(M,g,J)$ onto a Riemannian manifold $(N,g_{N})$. Then, for any $X,Y \in \mathcal{D}^{\theta_{1}}$ and $U,V \in \mathcal{D}^{\theta_{2}}$, we have
\begin{equation}\label{GauWei1}
  g(\nabla_{X}Y,U)=\csc^{2}\theta_{1}\,g(\mathcal{T}_{PU}FY-\mathcal{T}_{U}FPY+\mathcal{A}_{FU}FY,X),
\end{equation}
\begin{equation}\label{GauWei2}
  g(\nabla_{U}V,X)=\csc^{2}\theta_{2}\,g(\mathcal{T}_{PX}FV-\mathcal{T}_{X}FPV+\mathcal{A}_{FX}FV,U).
\end{equation}
\end{lemma}
\begin{proof}
  Assume that $X,Y$ be in $\mathcal{D}^{\theta_{1}}$ and $U,V$ be in $\mathcal{D}^{\theta_{2}}$. Then, from \eqref{e9}, \eqref{e10} and \eqref{decompvervec}, we have
  \begin{eqnarray*}
    g(\nabla_{X}Y,U)&=&g(\nabla_{X}JY,JU)\\
&=&g(\nabla_{X}PY,JU)+g(\nabla_{X}FY,JU).
  \end{eqnarray*}
  With the help of \eqref{e9} and \eqref{decompvervec}, we obtain
  \begin{eqnarray*}
 \Rightarrow g(\nabla_{X}Y,U)&=&-g(\nabla_{X}P^{2}Y,U)-g(\nabla_{X}FPY,U)\\
  &+&g(\nabla_{X}FY,PU)+g(\nabla_{X}FY,FY).
  \end{eqnarray*}
  By Lemma \ref{general}-(a), Remark \ref{remark1}, \eqref{testequation09} and \eqref{testequation11}, we get
  \begin{eqnarray*}
\Rightarrow g(\nabla_{X}Y,U)&=&\cos^{2}\theta_{1}\,g(\nabla_{X}Y,U)-g(\mathcal{T}_{X}FPY,U)\\
&+&g(\mathcal{T}_{X}FY,PU)+g(\mathcal{A}_{FY},FU).
  \end{eqnarray*}
If we edit the last equation and take into account the properties of O'Neill tensors $\mathcal{T}$ and $\mathcal{A}$, we get \eqref{GauWei1}. To obtain \eqref{GauWei2}, the same idea can be used.
\end{proof}
\subsection{Integrability}
In this section, we investigate the integrability of the distributions which are mentioned in the definition of bi-slant submersion.
\begin{theorem}
  Let $\pi$ be a bi-slant submersion from a Kaehlerian manifold $(M,g,J)$ onto a Riemannian manifold $(N,g_{N})$. Then, the slant distribution $\mathcal{D}^{\theta_{1}}$ is integrable if and only if
  \begin{equation*}
    g(\mathcal{T}_{PU}FY-\mathcal{T}_{U}FPY+\mathcal{A}_{FU}FY,X)=g(\mathcal{T}_{PU}FX-\mathcal{T}_{U}FPX+\mathcal{A}_{FU}FX,Y),
  \end{equation*}
  where $X,Y \in \mathcal{D}^{\theta_{1}}$ and $U \in \mathcal{D}^{\theta_{2}}$.
\end{theorem}
\begin{proof}
  Let $X,Y \in \mathcal{D}^{\theta_{1}}$ and $U \in \mathcal{D}^{\theta_{2}}$. Then, by \eqref{GauWei1}, we get
  \begin{eqnarray*}
g([X,Y],U)&=&g(\nabla_{X}Y,U)-g(\nabla_{Y}X,U)\\
&=&\csc^{2}\theta_{1}\big\{ g(\mathcal{T}_{PU}FY-\mathcal{T}_{U}FPY+\mathcal{A}_{FU}FY,X)\\
&-&g(\mathcal{T}_{PU}FX-\mathcal{T}_{U}FPX+\mathcal{A}_{FU}FX,Y)\big\}.
  \end{eqnarray*}
Therefore, the slant distribution $\mathcal{D}^{\theta_{1}}$ is integrable if and only if $[X,Y] \in \mathcal{D}^{\theta_{1}}$, for any $X,Y \in \mathcal{D}^{\theta_{1}}$. So we obtain the assertion.
\end{proof}
\begin{theorem}
  Let $\pi$ be a bi-slant submersion from a Kaehlerian manifold $(M,g,J)$ onto a Riemannian manifold $(N,g_{N})$. Then, the slant distribution $\mathcal{D}^{\theta_{2}}$ is integrable if and only if
  \begin{equation*}
    g(\mathcal{T}_{PX}FU-\mathcal{T}_{X}FPU+\mathcal{A}_{FX}FU,V)=g(\mathcal{T}_{PX}FV-\mathcal{T}_{X}FPV+\mathcal{A}_{FX}FV,U),
  \end{equation*}
  where $X \in \mathcal{D}^{\theta_{1}}$ and $U,V \in \mathcal{D}^{\theta_{2}}$.
\end{theorem}
\begin{proof}
   Let $X \in \mathcal{D}^{\theta_{1}}$ and $U,V \in \mathcal{D}^{\theta_{2}}$. Then, from \eqref{GauWei2}, we get
   \begin{eqnarray*}
g([U,V],X)&=&g(\nabla_{U}V,X)-g(\nabla_{V}U,X)\\
&=&\csc^{2}\theta_{2}\big\{ g(\mathcal{T}_{PX}FV-\mathcal{T}_{X}FPV+\mathcal{A}_{FX}FV,U)\\
&-&g(\mathcal{T}_{PX}FU-\mathcal{T}_{X}FPU+\mathcal{A}_{FX}FU,V)\big\}.
  \end{eqnarray*}
So, the assertion is obtained.
\end{proof}
\subsection{Totally and Mixed Geodesicness}
In this section, we investigate the geometry of the fibers, vertical distribution and horizontal distribution for a bi-slant submersion.
\begin{theorem}\label{GEOD1}
   Let $\pi$ be a bi-slant submersion from a Kaehlerian manifold $(M,g,J)$ onto a Riemannian manifold $(N,g_{N})$. Then, the slant distribution $\mathcal{D}^{\theta_{1}}$ defines a totally geodesic foliation on $ker\pi_{*}$ if and only if the following condition holds;
   \begin{equation}\label{geod1}
     g(\mathcal{T}_{PU}FY-\mathcal{T}_{U}FPY+\mathcal{A}_{FU}FY,X)=0,
   \end{equation}
   where $X,Y \in \mathcal{D}^{\theta_{1}}$ and $U \in \mathcal{D}^{\theta_{2}}$.
\end{theorem}
\begin{proof}
Let $X,Y \in \mathcal{D}^{\theta_{1}}$ and $U \in \mathcal{D}^{\theta_{2}}$. From \eqref{testequation09} and \eqref{GauWei1}, we have
\begin{eqnarray*}
g(\hat{\nabla}_{X}Y,U)&=&g(\nabla_{X}Y,U)\\
&=&\csc^{2}\theta_{1}\,g(\mathcal{T}_{PU}FY-\mathcal{T}_{U}FPY+\mathcal{A}_{FU}FY,X).
\end{eqnarray*}
So,  the slant distribution $\mathcal{D}^{\theta_{1}}$ defines a totally geodesic foliation on $ker\pi_{*}$ if and only if $\hat{\nabla}_{X}Y \in \mathcal{D}^{\theta_{1}}$ i.e. $g(\mathcal{T}_{PU}FY-\mathcal{T}_{U}FPY+\mathcal{A}_{FU}FY,X)$.
\end{proof}
\begin{theorem}\label{GEOD2}
   Let $\pi$ be a bi-slant submersion from a Kaehlerian manifold $(M,g,J)$ onto a Riemannian manifold $(N,g_{N})$. Then, the slant distribution $\mathcal{D}^{\theta_{2}}$ defines a totally geodesic foliation on $ker\pi_{*}$ if and only if the following condition holds;
    \begin{equation}\label{geod2}
     g(\mathcal{T}_{PX}FV-\mathcal{T}_{X}FPV+\mathcal{A}_{FX}FV,U)=0,
   \end{equation}
   where $X \in \mathcal{D}^{\theta_{1}}$ and $U,V \in \mathcal{D}^{\theta_{2}}$.
\end{theorem}
\begin{proof}
Let $X$ be in $\mathcal{D}^{\theta_{1}}$ and $U$ and $V$ be in $\mathcal{D}^{\theta_{2}}$. Thus, with the help of \eqref{testequation09} and \eqref{GauWei2}, we obtain
\begin{eqnarray*}
g(\hat{\nabla}_{U}V,X)&=&g(\nabla_{U}V,X)\\
&=&\csc^{2}\theta_{2}\,g(\mathcal{T}_{PX}FV-\mathcal{T}_{X}FPV+\mathcal{A}_{FX}FV,U).
\end{eqnarray*}
Therefore, we obtain the assertion.
\end{proof}
In the view of Theorem \ref{GEOD1} and Theorem \ref{GEOD2}, we have the following result.
\begin{corollary}
 Let $\pi$ be a bi-slant submersion from a Kaehlerian manifold $(M,g,J)$ onto a Riemannian manifold $(N,g_{N})$. Then, the vertical distribution $ker\pi_{*}$ is a locally product $M_{\mathcal{D}^{\theta_{1}}}\times M_{\mathcal{D}^{\theta_{2}}}$ if and only if \eqref{geod1} and \eqref{geod2} hold, where $M_{\mathcal{D}^{\theta_{1}}}$ and $M_{\mathcal{D}^{\theta_{2}}}$ are integral manifolds of the distributions $\mathcal{D}^{\theta_{1}}$ and $\mathcal{D}^{\theta_{2}}$, respectively.
\end{corollary}
\begin{theorem}\label{VERGEODESIC}
 Let $\pi$ be a bi-slant submersion from a Kaehlerian manifold $(M,g,J)$ onto a Riemannian manifold $(N,g_{N})$. Then, $ker\pi_{*}$ defines a totally geodesic foliation if and only if
 \begin{equation}\label{vergeodesic}
\omega (\mathcal{T}_{W}PZ+\mathcal{A}_{FZ}W)+F(\hat{\nabla}_{W}PZ+\mathcal{T}_{W}FZ)=0,\\
\end{equation}
where $W,Z \in ker\pi_{*}$.
\end{theorem}
\begin{proof}
  Let $W$ and $Z$ be in $ker\pi_{*}$. Then, from \eqref{testequation09}, \eqref{testequation11}, \eqref{e9}, \eqref{decompvervec} and \eqref{decomphorvec}, we obtain
  \begin{eqnarray*}
\nabla_{W}Z&=&-J\nabla_{W}JZ=-J(\nabla_{W}PZ+\nabla_{W}FZ)\\
&=&-J(\mathcal{T}_{W}PZ+\hat{\nabla}_{W}PZ+\mathcal{T}_{W}FZ+\mathcal{A}_{FZ}W)\\
&=&-\phi \mathcal{T}_{W}PZ-\omega \mathcal{T}_{W}PZ-P\hat{\nabla}_{W}PZ-F\hat{\nabla}_{W}PZ\\
& &-P\mathcal{T}_{W}FZ+F\mathcal{T}_{W}FZ-\phi \mathcal{A}_{FZ}W-\omega \mathcal{A}_{FZ}W.
  \end{eqnarray*}
  Thus, it is known that $ker\pi_{*}$ defines a totally geodesic foliation if and only if $\nabla_{W}Z \in ker\pi_{*}$. So, we get the assertion.
\end{proof}
\begin{theorem}\label{HORGEODESIC}
 Let $\pi$ be a bi-slant submersion from a Kaehlerian manifold $(M,g,J)$ onto a Riemannian manifold $(N,g_{N})$. Then, $ker\pi_{*}^{\perp}$ defines a totally geodesic foliation if and only if
 \begin{equation}\label{horgeodesic}
   \phi (\mathcal{A}_{\xi}\phi \eta +\mathcal{H}\nabla_{\xi}\omega \eta)+P(\mathcal{A}_{\xi}\omega \eta+\mathcal{V}\nabla_{\xi}\phi \eta)=0
 \end{equation}
 for any $\xi , \eta \in ker\pi_{*}^{\perp}$.
\end{theorem}
\begin{proof}
  Let $\xi , \eta \in ker\pi_{*}^{\perp}$. With the help of the equations \eqref{testequation}, \eqref{testequation111}, \eqref{e9}, \eqref{decompvervec} and \eqref{decomphorvec}, we get
  \begin{eqnarray*}
\nabla_{\xi}\eta&=&-J\nabla_{\xi}J\eta=-J(\nabla_{\xi}\phi \eta+\nabla_{\xi}\omega \eta)\\
&=&-J(\mathcal{A}_{\xi} \phi \eta +\mathcal{V}\nabla_{\xi}\phi \eta +\mathcal{H}\nabla_{\xi}\omega \eta+\mathcal{A}_{\xi}\omega \eta)\\
&=& -\phi \mathcal{A}_{\xi} \phi \eta - \omega \mathcal{A}_{\xi} \phi \eta - P\mathcal{V}\nabla_{\xi}\phi \eta -F\mathcal{V}\nabla_{\xi}\phi \eta\\
&-& \phi \mathcal{H}\nabla_{\xi}\omega \eta - \omega \mathcal{H}\nabla_{\xi}\omega \eta - P \mathcal{A}_{\xi}\omega \eta- F \mathcal{A}_{\xi}\omega \eta.
  \end{eqnarray*}
  Therefore, from the last equation, $ker\pi_{*}^{\perp}$ defines a totally geodesic foliation if and only if $\phi (\mathcal{A}_{\xi}\phi \eta +\mathcal{H}\nabla_{\xi}\omega \eta)+P(\mathcal{A}_{\xi}\omega \eta+\mathcal{V}\nabla_{\xi}\phi \eta)=0$.
\end{proof}
In the view of Theorem \ref{VERGEODESIC} and Theorem \ref{HORGEODESIC}, we give the following result.
\begin{corollary}
Let $\pi$ be a bi-slant submersion from a Kaehlerian manifold $(M,g,J)$ onto a Riemannian manifold $(N,g_{N})$. Then, the following three facts are equal to each other:
\begin{eqnarray*}
  &\textbf{(i)}& \text{M is a locally product }  M_{ker\pi_{*}} \times M_{ker\pi_{*}^{\perp}}, \\
   &\textbf{(ii)}& \pi \text{ is a totally geodesic map}, \\
  &\textbf{(iii)}& \eqref{vergeodesic} \text{ and } \eqref{horgeodesic} \text{ hold}, \\
\end{eqnarray*}
where $M_{ker\pi_{*}}$ and $M_{ker\pi_{*}^{\perp}}$ are integral manifolds of distributions $ker\pi_{*}$ and $ker\pi_{*}^{*}$, respectively.
\end{corollary}
\subsection{Parallelism of Canonical Structures}
In this section, we investigate the parallelism of the canonical structures for a bi-slant submersion.\\
Let $\pi$ be a bi-slant submersion from a Kaehlerian manifold $(M,g,J)$ onto a Riemannian manifold $(N,g_{N})$. Then, we define
\begin{eqnarray}
(\nabla_{W}P)Z&=&\hat{\nabla}_{W}PZ-P\hat{\nabla}_{W}Z, \label{pparallel}\\
(\nabla_{W}F)Z&=&\mathcal{H}\nabla_{W}FZ-F\hat{\nabla}_{W}Z,\label{fparallel}\\
(\nabla_{W}\phi)\xi &=&\hat{\nabla}_{W}\phi \xi-\phi \mathcal{H}\nabla_{W}\xi, \label{fiparallel}\\
(\nabla_{W}\omega)\xi &=&\mathcal{H}\nabla_{W}\omega \xi - \omega \mathcal{H}\nabla_{W}\xi, \label{omegaparallel}
\end{eqnarray}
where $W,Z \in ker\pi_{*}$ and $\xi \in ker\pi_{*}^{\perp}$. Then, it is said that
\begin{itemize}
  \item $P$ is \textit{parallel} $\Leftrightarrow$ $\nabla P\equiv 0$,
  \item $F$ is \textit{parallel} $\Leftrightarrow$ $\nabla F\equiv 0$,
  \item $\phi$ is \textit{parallel} $\Leftrightarrow$ $\nabla \phi \equiv 0$,
  \item $\omega$ is \textit{parallel} $\Leftrightarrow$ $\nabla \omega \equiv 0$.
\end{itemize}
In the view of Lemma \ref{lemmagenel} and \eqref{pparallel}$\sim$\eqref{omegaparallel}, we have the following lemma.
\begin{lemma}\label{paralleldefn}
Let $\pi$ be a bi-slant submersion from a Kaehlerian manifold $(M,g,J)$ onto a Riemannian manifold $(N,g_{N})$. Then, for any $W,Z \in ker\pi_{*}$ and $\xi \in ker\pi_{*}^{\perp}$, we get
\begin{eqnarray}
(\nabla_{W}P)Z&=&\phi \mathcal{T}_{W}Z-\mathcal{T}_{W}FZ,\label{pparallel2}\\
(\nabla_{W}F)Z&=&\omega \mathcal{T}_{W}Z-\mathcal{T}_{W}PZ, \label{fparallel2}\\
(\nabla_{W}\phi)\xi &=&P\mathcal{T}_{W}\xi - \mathcal{T}_{W}\omega \xi,\label{fiparallel2}\\
(\nabla_{W}\omega)\xi &=& F \mathcal{T}_{W}\xi - \mathcal{T}_{W}\phi \xi\label{omegaparallel2}.
\end{eqnarray}
\end{lemma}
\begin{theorem}
Let $\pi$ be a bi-slant submersion from a Kaehlerian manifold $(M,g,J)$ onto a Riemannian manifold $(N,g_{N})$. Then, $F$ is parallel if and only if $\phi$ is parallel.
\end{theorem}
\begin{proof}
  Let $F$ be parallel. Then, for any $W,Z \in ker\pi_{*}$, from \eqref{fiparallel2} we have $\omega \mathcal{T}_{W}Z=\mathcal{T}_{W}PZ$. By using \eqref{e9}, \eqref{decompvervec} and fundamental properties of O'Neill tensor $\mathcal{T}$, we get
  \begin{eqnarray*}
g(P\mathcal{T}_{W}\xi,Z)&=&g(J\mathcal{T}_{W}\xi,Z)=-g(\mathcal{T}_{W}\xi,JZ)\\
&=&-g(\mathcal{T}_{W}\xi,PZ)=g(\mathcal{T}_{W}PZ, \xi).
  \end{eqnarray*}
In the view of the fact of parallelism of $F$, we obtain
\begin{eqnarray*}
g(P\mathcal{T}_{W}\xi,Z)&=&g(\mathcal{T}_{W}PZ, \xi)=g(\omega \mathcal{T}_{W}Z, \xi)\\
&=&g(J\mathcal{T}_{W}Z, \xi)=-g(\mathcal{T}_{W}Z, \omega \xi)=g(\mathcal{T}_{W} \omega \xi,Z).
\end{eqnarray*}
So, we have for any $Z \in ker\pi_{*}$ $g(P\mathcal{T}_{W}\xi,Z)=g(\mathcal{T}_{W} \omega \xi,Z)$ i.e. $\phi$ is parallel.
\end{proof}
It is said that the fiber is \textit{$\mathcal{D}^{\theta_{1}}\!-\!\mathcal{D}^{\theta_{2}}$-mixed geodesic}, for any two distributions $\mathcal{D}^{\theta_{1}}$ and $\mathcal{D}^{\theta_{2}}$ defined on the fiber of a Riemannian submersion, if for any $X \in \mathcal{D}^{\theta_{1}}$ and $U \in \mathcal{D}^{\theta_{2}}$, $\mathcal{T}_{X}U=0$.
\begin{theorem}
Let $\pi$ be a bi-slant submersion from a Kaehlerian manifold $(M,g,J)$ onto a Riemannian manifold $(N,g_{N})$ with parallel canonical structure $F$. Then, the fibers are $\mathcal{D}^{\theta_{1}}\!-\!\mathcal{D}^{\theta_{2}}$-mixed geodesic.
\end{theorem}
\begin{proof}
  Let $X$ be in $\mathcal{D}^{\theta_{1}}$ and $U$ in $\mathcal{D}^{\theta_{2}}$. Then, from Lemma \ref{general}-(b) and \eqref{fparallel2}, we obtain
  \begin{equation*}
    \omega^{2}\mathcal{T}_{X}U=\omega(\omega \mathcal{T}_{X}U)=\omega \mathcal{T}_{X}PU=\mathcal{T}_{X}P^{2}U=-\cos^{2}\theta_{2}\mathcal{T}_{X}U.
  \end{equation*}
  On the other hand, from Lemma \ref{general}-(a) and \eqref{fparallel2}, we get
  \begin{equation*}
  \omega^{2}\mathcal{T}_{X}U=\omega^{2}\mathcal{T}_{U}X=\omega(\mathcal{T}_{U}PX)=\mathcal{T}_{U}P^{2}X=-\cos^{2}\theta_{1}\mathcal{T}_{U}X.
  \end{equation*}
  Therefore, we obtain
  \begin{equation*}
    -\cos^{2}\theta_{2}\mathcal{T}_{X}U=-\cos^{2}\theta_{1}\mathcal{T}_{X}U.
  \end{equation*}
Since $\cos^{2}\theta_{2}\mathcal{T}_{X}U=\cos^{2}\theta_{1}\mathcal{T}_{X}U$, we have $\mathcal{T}_{X}U=$. That implies the fibers are $\mathcal{D}^{\theta_{1}}\!-\!\mathcal{D}^{\theta_{2}}$-mixed geodesic.
\end{proof}
\section{Curvature Relations}
In this section, the sectional curvatures of the total space, base space and the fibers of a bi-slant submersion are investigated.\\
 Let $\pi$ be a bi-slant submersion from a Kaehlerian manifold $(M,g,J)$ onto a Riemannian manifold $(N,g_{N})$. We denote the Riemannian curvature tensors of $M$, $N$ and any fiber of the submersion with $R$, $R^{*}$ and $\hat{R}$, respectively. Also, we denote the sectional curvatures of $M$, $N$ and any fiber of the submersion with $K$, $K^{*}$ and $\hat{K}$, respectively. It is known that the sectional curvature for a Riemannian submersion is defined, for any pair of non-zero orthogonal vectors $U$ and $V$ \cite{O}
\begin{equation}\label{sectionalcurvature}
K(U,V)= \frac{R(U,V,V,U)}{g(U,U)g(V,V)}.
\end{equation}
For any $e_{1},e_{2}\in ker\pi_{*}$ and $E_{1}, E_{2} \in ker\pi^{\perp}_{*} $ the Riemannian curvature tensor $R$ is given by \cite{O}
\begin{eqnarray}
R(e_{1},e_{2},e_{3},e_{4})&=&\hat{R}(e_{1},e_{2},e_{3},e_{4})-g(\mathcal{T}_{e_{1}}e_{4},\mathcal{T}_{e_{2}}e_{3})\nonumber \\
& &+g(\mathcal{T}_{e_{2}}e_{4},\mathcal{T}_{e_{1}}e_{3}), \label{r1}
\end{eqnarray}
\begin{eqnarray}
R(e_{1},e_{2},e_{3},E_{1})&=&g((\nabla_{e_{1}}\mathcal{T})(e_{2},e_{3}),E_{1})-g((\nabla_{e_{2}}\mathcal{T})(e_{1},e_{3}),E_{1}),\label{r2}
\end{eqnarray}
\begin{eqnarray}
R(E_{1},E_{2},E_{3},e_{1})&=&-g((\nabla_{E_{3}}\mathcal{A})(E_{1},E_{2}),e_{1})-g(\mathcal{A}_{E_{1}}E_{2},\mathcal{T}_{e_{1}}E_{3})\nonumber \\
& &g(\mathcal{A}_{E_{2}}E_{3},\mathcal{T}_{e_{1}}E_{1})+g(\mathcal{A}_{E_{3}}E_{1},\mathcal{T}_{e_{1}}E_{2}),\label{r3}
\end{eqnarray}
\begin{eqnarray}
R(E_{1},E_{2},E_{3},E_{4})&=&R^{*}(E_{1},E_{2},E_{3},E_{4})+2g(\mathcal{A}_{E_{1}}E_{2},\mathcal{A}_{E_{3}}E_{4})\nonumber \\
& &-g(\mathcal{A}_{E_{2}}E_{3},\mathcal{A}_{E_{1}}E_{4})+g(\mathcal{A}_{E_{1}}E_{3},\mathcal{A}_{E_{2}}E_{4}),\label{r4}
\end{eqnarray}
\begin{eqnarray}
R(E_{1},E_{2},e_{1},e_{2})&=&-g((\nabla_{e_{1}}\mathcal{A})(E_{1},E_{2}),e_{2})+g((\nabla_{e_{2}}\mathcal{A})(E_{1},E_{2}),e_{1})\nonumber \\
& &-g(\mathcal{A}_{E_{1}}e_{1},\mathcal{A}_{E_{2}}e_{2})+g(\mathcal{A}_{E_{1}}e_{2},\mathcal{A}_{E_{2}}e_{1})\nonumber \\
& &+g(\mathcal{T}_{e_{1}}E_{1},\mathcal{T}_{e_{2}}E_{2})-g(\mathcal{T}_{e_{2}}E_{1},\mathcal{T}_{e_{1}}E_{2}),\label{r5}
\end{eqnarray}
\begin{eqnarray}
R(E_{1},e_{1},E_{2},e_{2})&=&-g((\nabla_{E_{1}}\mathcal{T})(e_{1},e_{2}),E_{2})-g((\nabla_{e_{1}}\mathcal{A})(E_{1},E_{2}),e_{2})\nonumber \\
& & g(\mathcal{T}_{e_{1}}E_{1},\mathcal{T}_{e_{2}}E_{2})-g(\mathcal{A}_{E_{1}}e_{1},\mathcal{A}_{E_{2}}e_{2}),\label{r6}
\end{eqnarray}
where $R$, $R^{*}$ and $\hat{R}$ is Riemannian curvature of $M$, $N$ and fiber, respectively.\\
Furthermore, let $\pi$ be submersion from a Riemannian manifold $M$ onto a Riemannian manifold $N$. Then, the followings are given \cite{O}:
\begin{eqnarray}
K(e_{1},e_{2})=\hat{K}(e_{1},e_{2})-g(\mathcal{T}_{e_{1}}e_{1},\mathcal{T}_{e_{2}}e_{2})+\|\mathcal{T}_{e_{1}}e_{2}\|^{2},\label{curv1}
\end{eqnarray}
\begin{eqnarray}
K(E_{1},e_{1})=g((\nabla_{E_{1}}\mathcal{T})(e_{1},e_{1}),E_{1})+\|\mathcal{A}_{E_{1}}e_{1}\|^{2}-\|\mathcal{T}_{e_{1}}E_{1}\|^{2},\label{curv2}
\end{eqnarray}
\begin{eqnarray}
K(E_{1},E_{2})=K^{*}(E_{1},E_{2})-3\|\mathcal{A}_{E_{1}}E_{2}\|^{2},\label{curv3}
\end{eqnarray}
where $e_{1},e_{2}\in ker\pi_{*}$ and $E_{1}, E_{2} \in ker\pi^{\perp}_{*} $ orthonormal vector fields.
 \begin{theorem}
 Let $\pi$ be a bi-slant submersion from a Kaehlerian manifold $(M,g,J)$ onto a Riemannian manifold $(N,g_{N})$. Then, we obtain
 \begin{eqnarray}\label{sect1}
K(e_{1},e_{2})&=&\hat{K}(Pe_{1},Pe_{2})\|Pe_{1}\|^{-2}\|Pe_{2}\|^{-2}+K^{*}(F{e_1},Fe_{2})\|Fe_{1}\|^{-2}\|Fe_{2}\|^{-2}\nonumber\\
& &-g(\mathcal{T}_{Pe_{1}}Pe_{1},\mathcal{T}_{Pe_{2}}Pe_{2})+\|\mathcal{A}_{Fe_{1}}Pe_{2}\|^{2}\nonumber\\
& &+g((\nabla_{Fe_{2}}\mathcal{T})(Pe_{1},Pe_{2}),Fe_{2})-\|\mathcal{T}_{Pe_{1}}Fe_{2}\|^{2}\nonumber\\
& &-3\|\mathcal{A}_{Fe_{1}}Fe_{2}\|^{2}+\|\mathcal{T}_{Pe_{2}}Pe_{1}\|^{2},
 \end{eqnarray}
 \begin{eqnarray}\label{sect2}
K(e_{1},E_{1})&=&\hat{K}(Pe_{1},\phi E_{1})\|Pe_{1}\|^{-2}\|\phi E_{1}\|^{-2}+K^{*}(Fe_{1},\omega E_{1})\|Fe_{1}\|^{-2}\|\omega E_{1}\|^{-2}\nonumber\\
& &-\|\mathcal{T}_{\phi E_{1}}Pe_{1}\|^{2}-\|\mathcal{T}_{P e_{1}}\omega E_{1}\|^{2}-3\|\mathcal{A}_{Fe_{1}}\omega E_{1}\|^{2} \nonumber \\
& &+ \|\mathcal{A}_{\omega E_{1}}Pe_{1}\|^{2}-\|\mathcal{T}_{\phi E_{1}}Fe_{1}\|^{2}-g(\mathcal{T}_{Pe_{1}}Pe_{1},\mathcal{T}_{\phi E_{1}}\phi E_{1}) \nonumber \\
& & +\|\mathcal{A}_{Fe_{1}}\phi E_{1}\|^{2}+g((\nabla_{\omega E_{1}}\mathcal{T})(Pe_{1},Pe_{1}),\omega E_{1}) \nonumber \\
& & +g((\nabla_{Fe_{1}}\mathcal{T})(\phi E_{1},\phi E_{1}),Fe_{1}),
\end{eqnarray}
\begin{eqnarray}\label{sect3}
K(E_{1},E_{2})&=&\hat{K}(\phi E_{1},\phi E_{2})\|\phi E_{1}\|^{-2}\|\phi E_{2}\|^{-2}+K^{*}(\omega E_{1},\omega E_{2})\|\omega E_{1}\|^{-2}\|\omega E_{2}\|^{-2}\nonumber \\
& &+\|\mathcal{T}_{\phi E_{2}}\phi E_{1}\|^{2}-g(\mathcal{T}_{\phi E_{1}}\phi E_{1},\mathcal{T}_{\phi E_{2}}\phi E_{2}) \nonumber \\
& & +g((\nabla_{\omega E_{2}}\mathcal{T})(\phi E_{1},\phi E_{2}),\omega E_{2})-\|\mathcal{T}_{\phi E_{1}}\omega E_{2}\|^{2}\nonumber \\
& & +\|\mathcal{A}_{\omega E_{2}}\phi E_{1}\|^{2}+g((\nabla_{\omega E_{1}}\mathcal{T})(\phi E_{2},\phi E_{2}),\omega E_{1})\nonumber \\
& & -\|\mathcal{T}_{\phi E_{2}}\omega E_{1}\|^{2}+\|\mathcal{A}_{\omega E_{1}}\phi E_{2}\|^{2}-3\|\mathcal{A}_{\omega E_{1}}\omega E_{2}\|^{2}.
\end{eqnarray}
\end{theorem}
\begin{proof}
Let $e_{1},e_{2}\in ker\pi_{*}$ and $E_{1}, E_{2} \in ker\pi^{\perp}_{*} $ be orthonormal vector fields. Then, by the fact that $K(e_{1},e_{2})=K(Je_{1},Je_{2})$, \eqref{decompvervec} and \eqref{decomphorvec}, we get
\begin{eqnarray*}
K(e_{1},e_{2})=K(Je_{1},Je_{2})&=&K(Pe_{1},Pe_{2})+K(Pe_{1},Fe_{2})\\
& &+K(Fe_{1},Pe_{2})+K(Fe_{1},Fe_{2}).
\end{eqnarray*}
By the definition of the sectional curvature, we obtain
\begin{eqnarray*}
\Rightarrow K(e_{1},e_{2})&=& R(Pe_{1},Pe_{2},Pe_{2},Pe_{1})+R(Pe_{1},Fe_{2},Fe_{2},Pe_{1})\\
& &R(Fe_{1},Pe_{2},Pe_{2},Fe_{1})+R(Fe_{1},Fe_{2},Fe_{2},Fe_{1}).
\end{eqnarray*}
Thus, with the help of \eqref{r1}$\sim$\eqref{r6}, we have
\begin{eqnarray*}
\Rightarrow K(e_{1},e_{2})&=&\hat{R}(Pe_{1},Pe_{2},Pe_{2},Pe_{1})-g(\mathcal{T}_{Pe_{1}}Pe_{1},\mathcal{T}_{Pe_{2}}Pe_{2})+\|\mathcal{T}_{Pe_{1}}Pe_{2}\|^{2}\\
& & +g((\nabla_{Fe_{2}}\mathcal{T})(Pe_{1},Pe_{1}),Fe_{2})-\|\mathcal{T}_{Pe_{1}}Fe_{2}\|^{2}+\|\mathcal{A}_{Fe_{2}}Pe_{1}\|^{2}\\
& &+g((\nabla_{Fe_{1}}\mathcal{T})(Pe_{2},Pe_{2}),Fe_{1})-\|\mathcal{T}_{Pe_{2}}Fe_{1}\|^{2}+\|\mathcal{A}_{Fe_{1}}Pe_{2}\|^{2}\\
& & +R^{*}(Pe_{1},Pe_{2},Pe_{2},Pe_{1})-3\|\mathcal{A}_{Fe_{1}}Fe_{2}\|^{2}.
\end{eqnarray*}
Since,
\begin{equation*}
\hat{R}(Pe_{1},Pe_{2},Pe_{2},Pe_{1})=\hat{K}(Pe_{1},Pe_{2})\|Pe_{1}\|^{-2}\|Pe_{2}\|^{-2}
\end{equation*}
\[and\]
\begin{equation*}
R^{*}(Pe_{1},Pe_{2},Pe_{2},Pe_{1})=K^{*}(F{e_1},Fe_{2})\|Fe_{1}\|^{-2}\|Fe_{2}\|^{-2}
\end{equation*}
\eqref{sect1} is obtained. \eqref{sect2} and \eqref{sect3} can be obtained with a similar way.
\end{proof}
Now, we give some inequalities for sectional curvatures of total manifold, base manifold and fibers.
\begin{corollary}
 Let $\pi$ be a bi-slant submersion from a Kaehlerian manifold $(M,g,J)$ onto a Riemannian manifold $(N,g_{N})$. Then, we have\\
\[
  \begin{array}{ccc}
    \hat{K}(Pe_{1},Pe_{2})\|Pe_{1}\|^{-2}\|Pe_{2}\|^{-2} &   &g(\mathcal{T}_{Pe_{1}}Pe_{1},\mathcal{T}_{Pe_{2}}Pe_{2}) \\
   +K^{*}(F{e_1},Fe_{2})\|Fe_{1}\|^{-2}\|Fe_{2}\|^{-2}& \leq &  +\|\mathcal{T}_{Pe_{1}}Fe_{2}\|^{2}, \\
    -\hat{K}(e_{1},e_{2}) &   &   \\
 \end{array}
\]
\end{corollary}
\begin{proof}
  Let $e_{1},e_{2}\in ker\pi_{*}$ be orthonormal vector fields. Then, by \eqref{curv1} and \eqref{sect1}, we get
  \begin{eqnarray*}
\hat{K}(e_{1},e_{2})-g(\mathcal{T}_{e_{1}}e_{1},\mathcal{T}_{e_{2}}e_{2})+\|\mathcal{T}_{e_{1}}e_{2}\|^{2}&=&\hat{K}(Pe_{1},Pe_{2})\|Pe_{1}\|^{-2}\|Pe_{2}\|^{-2}\nonumber \\
& &+K^{*}(F{e_1},Fe_{2})\|Fe_{1}\|^{-2}\|Fe_{2}\|^{-2} \nonumber \\
& & -g(\mathcal{T}_{Pe_{1}}Pe_{1},\mathcal{T}_{Pe_{2}}Pe_{2})\nonumber \\
& &+\|\mathcal{A}_{Fe_{1}}Pe_{2}\|^{2} -\|\mathcal{T}_{Pe_{1}}Fe_{2}\|^{2}\nonumber \\
& &+g((\nabla_{Fe_{2}}\mathcal{T})(Pe_{1},Pe_{2}),Fe_{2})\nonumber \\
& & -3\|\mathcal{A}_{Fe_{1}}Fe_{2}\|^{2}+\|\mathcal{T}_{Pe_{2}}Pe_{1}\|^{2}.
  \end{eqnarray*}
Thus, we obtain the assertion.
\end{proof}
\begin{corollary}
 Let $\pi$ be a bi-slant submersion from a Kaehlerian manifold $(M,g,J)$ onto a Riemannian manifold $(N,g_{N})$. Then,
\[
   \begin{array}{ccc}
     \hat{K}(Pe_{1},\phi E_{1}) &   & g((\nabla_{E_{1}}\mathcal{T})(e_{1},e_{1}),E_{1})+\|\mathcal{A}_{E_{1}}e_{1}\|^{2}\\
     +K^{*}(Fe_{1},\omega E_{1}) & \leq &+\|\mathcal{T}_{P e_{1}}\omega E_{1}\|^{2}+\|\mathcal{T}_{\phi E_{1}}Pe_{1}\|^{2}  \\
       &   & +3\|\mathcal{A}_{Fe_{1}}\omega E_{1}\|^{2}+g(\mathcal{T}_{Pe_{1}}Pe_{1},\mathcal{T}_{\phi E_{1}}\phi E_{1}), \\
   \end{array}
 \]
 where $e_{1}\in ker\pi_{*}$ and $E_{1} \in ker\pi^{\perp}_{*} $ orthonormal vector fields.
\end{corollary}
\begin{proof}
  Let $e_{1}\in ker\pi_{*}$ and $E_{1} \in ker\pi^{\perp}_{*} $ be orthonormal vector fields. Then, by \eqref{curv2} and \eqref{sect2}, we have
  \begin{eqnarray*}
g((\nabla_{E_{1}}\mathcal{T})(e_{1},e_{1}),E_{1})+\|\mathcal{A}_{E_{1}}e_{1}\|^{2}&=&\hat{K}(Pe_{1},\phi E_{1})\|Pe_{1}\|^{-2}\|\phi E_{1}\|^{-2}\nonumber \\
& & +K^{*}(Fe_{1},\omega E_{1})\|Fe_{1}\|^{-2}\|\omega E_{1}\|^{-2} \nonumber \\
& & -\|\mathcal{T}_{\phi E_{1}}Pe_{1}\|^{2}-\|\mathcal{T}_{P e_{1}}\omega E_{1}\|^{2}\nonumber \\
& &-3\|\mathcal{A}_{Fe_{1}}\omega E_{1}\|^{2}+\|\mathcal{A}_{\omega E_{1}}Pe_{1}\|^{2} \nonumber \\
& &-\|\mathcal{T}_{\phi E_{1}}Fe_{1}\|^{2}-g(\mathcal{T}_{Pe_{1}}Pe_{1},\mathcal{T}_{\phi E_{1}}\phi E_{1})\nonumber \\
& &+\|\mathcal{A}_{Fe_{1}}\phi E_{1}\|^{2}+\|\mathcal{T}_{e_{1}}E_{1}\|^{2}\nonumber \\
& &+g((\nabla_{\omega E_{1}}\mathcal{T})(Pe_{1},Pe_{1}),\omega E_{1}) \nonumber \\
& &+g((\nabla_{Fe_{1}}\mathcal{T})(\phi E_{1},\phi E_{1}),Fe_{1}).
  \end{eqnarray*}
Therefore, the assertion is obtained.
\end{proof}
\begin{corollary}
Let $\pi$ be a bi-slant submersion from a Kaehlerian manifold $(M,g,J)$ onto a Riemannian manifold $(N,g_{N})$. Then, we obtain
\[
  \begin{array}{ccc}
    \hat{K}(\phi E_{1},\phi E_{2})\|\phi E_{1}\|^{-2}\|\phi E_{2}\|^{-2} &   & g(\mathcal{T}_{\phi E_{1}}\phi E_{1},\mathcal{T}_{\phi E_{2}}\phi E_{2})+\|\mathcal{T}_{\phi E_{1}}\omega E_{2}\|^{2} \\
    +K^{*}(\omega E_{1},\omega E_{2})\|\omega E_{1}\|^{-2}\|\omega E_{2}\|^{-2} & \leq & + \|\mathcal{T}_{\phi E_{2}}\omega E_{1}\|^{2}+3\|\mathcal{A}_{\omega E_{1}}\omega E_{2}\|^{2} \\
    -K^{*}(E_{1},E_{2}) &   &  \\
  \end{array}
\]
\end{corollary}
\begin{proof}
  Let $E_{1}, E_{2} \in ker\pi^{\perp}_{*} $ be orthonormal vector fields. From \eqref{curv3} and \eqref{sect3}, we get
 \begin{eqnarray*}
K^{*}(E_{1},E_{2})-3\|\mathcal{A}_{E_{1}}E_{2}\|^{2}&=&\hat{K}(\phi E_{1},\phi E_{2})\|\phi E_{1}\|^{-2}\|\phi E_{2}\|^{-2}\nonumber \\
& &+K^{*}(\omega E_{1},\omega E_{2})\|\omega E_{1}\|^{-2}\|\omega E_{2}\|^{-2}\nonumber \\
& &+\|\mathcal{T}_{\phi E_{2}}\phi E_{1}\|^{2}-g(\mathcal{T}_{\phi E_{1}}\phi E_{1},\mathcal{T}_{\phi E_{2}}\phi E_{2})\nonumber \\
& &+g((\nabla_{\omega E_{2}}\mathcal{T})(\phi E_{1},\phi E_{2}),\omega E_{2})-\|\mathcal{T}_{\phi E_{1}}\omega E_{2}\|^{2}\nonumber \\
& &+\|\mathcal{A}_{\omega E_{2}}\phi E_{1}\|^{2}+g((\nabla_{\omega E_{1}}\mathcal{T})(\phi E_{2},\phi E_{2}),\omega E_{1}) \nonumber \\
& &-\|\mathcal{T}_{\phi E_{2}}\omega E_{1}\|^{2}+\|\mathcal{A}_{\omega E_{1}}\phi E_{2}\|^{2}-3\|\mathcal{A}_{\omega E_{1}}\omega E_{2}\|^{2}.
 \end{eqnarray*}
Hence, the assertion is obtained.
\end{proof}


\end{document}